\documentclass[12pt]{article}
\UseRawInputEncoding 
\usepackage[all]{xy}
\usepackage{amsfonts}
\usepackage{mathrsfs}
\usepackage{amsfonts}
\usepackage{amsfonts}
\usepackage{amsfonts}
\usepackage{amsfonts}
\usepackage{amsfonts}
\usepackage{amsfonts}
\usepackage{amsfonts}
\usepackage{amsfonts}
\usepackage{amsfonts}
\usepackage{amsfonts}
\usepackage{amsfonts}
\usepackage{amsfonts}
\usepackage{amsfonts}
\usepackage{amsfonts}
\usepackage{amscd}
\usepackage{bbm}
\usepackage [dvips]{graphics}
\usepackage[centertags]{amsmath}
\usepackage{amsfonts}
\usepackage{amssymb}

\usepackage{xypic}
\usepackage{amsthm}
\usepackage{newlfont}
\usepackage{stmaryrd}
\usepackage[]{titletoc}  % from here
\usepackage{mathrsfs}
\usepackage{stmaryrd}
\usepackage [dvips]{graphics}
\usepackage [none]{hyphenat}
\usepackage[centertags]{amsmath}
\usepackage{amsfonts}
\usepackage{amssymb}
\usepackage{amsthm}
\usepackage{newlfont}
\usepackage{hyperref}

\newlength{\defbaselineskip}
\newcommand{\setlinespacing}[1]%
           {\setlength{\baselineskip}{#1 \defbaselineskip}}

\theoremstyle{plain}
\newtheorem{thm}{Theorem}[section]

\newtheorem{prop}[thm]{Proposition}

\newtheorem{rem}[thm]{Remark}
\newtheorem{Def}[thm]{Definition}

\newcommand{\cc}{\mathbb{C}}

\newcommand{\bin}{\binom}

\newcommand{\w}{\widetilde}

\newcommand{\p}{\partial}

  \newcommand{\sss}{\textbf{s}}

\makeatletter\@addtoreset{equation}{section} \makeatother
\begin{document}

\textwidth=13.5cm
  \textheight=23cm
  \hoffset=-1cm

  \baselineskip=17pt
\title { Pascal  algebra  of matrices    and Pascal map on jet bundles}
\author{Li Chen}
\date{}
\maketitle {\textbf{}
\def\zz{\textbf{z}}
\def\ww{\textbf{w}}
\def\zzz{\textbf{z}_0}
\def\bg{\mbox{\boldmath$\gamma$}}
\def\ba{\mbox{\boldmath$a$}}
\def\bb{\mbox{\boldmath$b$}}
\def\bet{\mbox{\boldmath$\beta$}}
\def\M{\mathcal{M}}
\def\H{\mathcal{H}}

%action of Ni; block-entry
%
%staring from 0th; right left action; A otimes In
%
%how to say NiNj* restriction
%
%widetilde N*
%
%Z=? submanifold
%
%intertwine Ni or Ti?
%
%definition of $H_Z^2$
%
%definition of Xij-curvature operator
%
%joint kernel
%
%Ez EkZ

\textbf{Abstract:} We identify and study       a matrix    algebra consisting of Pascal-type matrices.  The generator of the matrix algebra is shown to well define a canonical bundle map, called the Pascal map on  jet bundles,   and we use it to give an intrinsic definition of point-wise contact between Hermitian vector bundles in terms of  unitary equivalence of the Pascal maps.


 \section{Introduction}
An interesting  topic in differential geometry is to formulate intrinsic concepts on vector bundles out of extrinsic ones for smooth maps(sub-manifolds)  into   ambient spaces. This note concerns    geometric notion  on  high order behaviours carried   by jet bundles.  We   formulate   a    canonical bundle map, called the Pascal map, on   jet bundles and exhibit its use  in the  extrinsic-intrinsic  transition by refining the extrinsic notion of point-wise contact between smooth maps into an intrinsic one that makes sense on  vector bundles.

  In Section 2 we identify and study  a matrix algebra called the \emph{Pascal algebra} as a preparation, which is of independent interest in  extending the classical Pascal matrix as well as   later works on  generalized Pascal matrices. In Section 3 we show  that the generator of the Pascal   algebra represents a well-defined bundle map, called the \emph{Pascal map},  on the jet bundle of a given vector bundle.  In Section 4 we use the Pascal map to give  an  intrinsic definition  of point-wise contact between Hermitian vector bundles. Both Section 3 and Section 4 are confined in the setting of vector bundles over a domain in $\cc$, and in Section 5 we outline a several variable extension which works with domains in $\cc^m$, $m>1$.
 \section{Pascal   algebra}

Given a fixed positive integer $n$, we denote by   $\Lambda^n$ the set of  $(n+1)\times (n+1)$ matrices of the following form

\begin{equation}\label{pasc}
\left(\begin{array}{cccccc}{a_0}  & { }  &{ } & { } & { }\\ {a_1} & {a_0}  &  { } & {} & { }\\{a_2} & {2a_1} &  {a_0} & { } & { }\\{a_3} & {3a_2} &  {3a_1} & {a_0 } & { }\\  {\vdots} &{\vdots }  & { } & {}& {\ddots}& { } \\{a_n} & {\bin n 1 a_{n-1}}  &{\cdots} &{\cdots} & {\bin n {n-1}a_1}& {a_0 } \end{array}\right). \end{equation}  That is,  a  matrix  in $\Lambda^n$ is determined by  $n+1$ complex numbers   $a_0,a_1,\cdots a_n$ lying   in its first column,   whose $(i,j)$ entry is $\bin {i-1} {j-1} a_{i-j}$ if  $1\leq j\leq i\leq n+1$ and $0$ if $j>i$.

 Setting $a_0=a_1= \cdots =a_n=1$ yields the classical  Pascal matrix,    the study of which  starts from Call and Velleman \cite{CV}. Letting $a_n$   vary with $n$  according to different  rules(for instance, set $a_n=x^n$ for a indeterminate $x$),  one gets various  versions of ``generalized Pascal matrices"(\cite{BP,BT,YM,ZL,ZW}). The class  $\Lambda^n$ we identified above allows arbitrary first column entries hence  is the  ``biggest" class.

 \bigskip

\begin{thm}\label{pascal} Let
  $P$ be the matrix in $\Lambda^n$ whose first column is given by $a_1=1$ and $a_k=0,$  $k\neq 1$, that is,

\begin{equation}\label{21}
P=\left(\begin{array}{cccccc} {0}  & { }  &{ } & { } & { }\\ { 1} & { 0}  &  { } & {} & { }\\{0} & {2 } &  { 0} & { } & { }\\{0} & {0} &  {3 } & {0} & { }\\  {\vdots} &{\vdots }  & { } & {}& {\ddots}& { } \\{0} & {  0 }  &{\cdots} &{\cdots} & {n}& { 0 } \end{array}\right),\end{equation}  then $\Lambda^n$ is the commutant of $P$.

\end{thm}

\begin{proof} To see $\Lambda^n\subseteq P'$,  let $Q$ be in  $\Lambda^n$ and $(a_0,a_1,\cdots, a_n)^T$ be its first column, then both  $PQ$ and $QP$ are  lower triangular with zero diagonals, and it remains to compare their $(i,j)$ entries  for $j\leq i-1$. Observing that  the $(i,j)$ entry of $QP$
is $j\bin {i-1}{j}a_{i-j-1}$ and  the corresponding entry of $PQ$ is $(i-1)\bin {i-2}{j-1}a_{i-j-1}$,   the conclusion follows from the elementary identity $j\bin {i-1}{j}=(i-1)\bin {i-2}{j-1}$.

For the other direction,  let $Q$ be any $(n+1)\times(n+1)$ matrix  such that $PQ=QP$, and we have to  show $Q\in\Lambda^n$. To this end,    we  view $P$ and $Q$  as linear maps acting   on an $n+1$ dimensional space with a fixed base   $\{s_0,s_1,\cdots s_n\}$. Then $P$ corresponds to the action $Ps_0=0$ and $Ps_k=ks_{k-1}$,  $k=1,2,\cdots ,n$. Moreover,  if  $(a_0,a_1,\cdots, a_n)^T$ is the first column of $Q$, then  for every $0\leq k\leq n$,  $Qs_k=a_ks_0+$ other terms involving  $ s_1,\cdots s_n$.

Now it suffices to show if $PQ=QP$, then
\begin{equation}\label{mb}Qs_k=\sum_{i=0}^k \bin k ia_{k-i}s_i, k=0,1\cdots n \end{equation}

 Comparing  $PQs_0$ and $QPs_0$ one immediately sees that $Qs_0$ only has $s_0$ component, hence  (\ref{mb}) holds  for $k=0$
  and we  verify (\ref{mb}) by induction on $k$. Precisely, it suffices to show
\begin{equation}\label{mbz}Qs_{k+1}=a_{k+1}s_0+\sum_{i=1}^{k+1} \bin {k+1} ia_{k+1-i}s_i\end{equation} assuming (\ref{mb}).

To this end, write
$$Qs_{k+1}=a_{k+1}s_0+\sum_{i=1}^{n}b_is_i$$ for some coefficients $b_1,\cdots,b_{n}$, thus
$ PQs_{k+1}=\sum_{i=1}^{n}ib_is_{i-1}.$  While by (\ref{mb}) we have
$$QPs_{k+1}=(k+1)Qs_k=(k+1)\sum_{i=0}^k \bin k ia_{k-i}s_i=(k+1)\sum_{i=1}^{k+1} \bin k {i-1}a_{k+1-i}s_{i-1}.$$ Now $PQ=QP$ implies $b_i=0$ for $k+2\leq i\leq n$ and $b_i=\frac{(k+1)\bin k {i-1}a_{k+1-i}}{i}=\bin {k+1} ia_{k+1-i}$ for $1\leq i\leq k+1$, which gives (\ref{mbz}) as desired.

%(ii)The first part follows from necessity of (i). Moreover, for any $2\leq k\leq n$, let $Q_k$  be the matrix in $\Lambda^n$ whose first column is given by $a_k=1 $ and $ a_j=0,j\neq k$, then   $I,P,Q_2\cdots Q_n$ forms  a linear base of  $\Lambda^n$ and it suffices to show $Q_2,\cdots Q_n$ can be generated by $P$.
%
% In fact, with the notation of (i),   the action of $P^k$ is given by $P^ks_0=P^ks_1=\cdots P^ks_{k-1}=0,$ and $P^ks_k=k!s_0, P^ks_{k+1}=\frac{(k+1)!}{1!}s_1\cdots,P^ks_n=\frac{n!}{(n-k)!}s_{n-k}$. On the other hand, the action of $Q_k$ is, by its definition,   $Q_ks_0=Q_ks_1=\cdots Q_ks_{k-1}=0,$ and $Q_ks_k=s_0,Q_ks_{k+1}=\bin {k+1}{1}s_1,\cdots, Q_ks_{n}=\bin n {n-k}s_{n-k}$. Therefore, $P^k=k!Q_k$, completing the proof.

\end{proof}
An immediate consequence of Theorem \ref{pascal} is that $\Lambda^n$ is an algebra, which justifies the following

\begin{Def}The collection $\Lambda^n$  of   lower triangular matrices of the form (\ref{pasc}) is called the \textbf{Pascal algebra}.\end{Def}
\begin{rem} An alternative way to prove Theorem \ref{pascal} is to show both $\Lambda^n$ and $P'$ equals the algebra of polynomials in $P$(so $P$ is the generator of $\Lambda^n$) which involves arguments with minimal polynomials.  The proof we present above is a straightforward ``entry comparing" and is   easily seen to work in the setting of block matrices.
%Lemma (\ref{close}) asserts that $\Lambda^n$ is the commutant of $P$ in the algebra of $(n+1)\times (n+1)$ lower triangular matrices. In particular, we have
%\begin{cor}\label{close}
%The set $\Lambda^n$ is closed with respect to matrix multiplication and inverse.
 Precisely,
replacing the scalers $a_0,\cdots, a_n$ in (\ref{pasc}) by $l\times l$ matrices $A_0,\cdots A_n$,  one gets a collection of  $(n+1)\times (n+1)$ block matrices. The set of these block matrices is still denoted by $\Lambda^n$,   which equals  the commutant of
 \begin{equation}\label{22}
P=\left(\begin{array}{cccccc} {0}  & { }  &{ } & { } & { }\\ { I} & { 0}  &  { } & {} & { }\\
{0} & {2I } &  { 0} & { } & { }\\{0} & {0} &  {3I } & {0} & { }\\  {\vdots} &{\vdots }  & { } & {}& {\ddots}& { } \\
{0} & {  0 }  &{\cdots} &{\cdots} & {nI}& { 0 } \end{array}\right).\end{equation}
Later, $\Lambda^n\subseteq P'$ will be needed in Theorem \ref{pascc} and $P'\subseteq \Lambda^n$ will be needed in Proposition \ref{contin}.

%Now we turn back  to holomorphic jet bundles.
%that is, the transition matrix  for
%  $\{\textbf{s}, \textbf{s}'   \cdots, \textbf{s}^{(n)} \}$ and $\{\textbf{t}, \textbf{t}'   \cdots, \textbf{t}^{(n)} \}$, as frames of  the jet bundle $E^n$,  belongs to  $\Lambda^n$ whose  first column is $A,A',\cdots A^{(n)}$, and we briefly denote this block transition matrix   by  $\Lambda^n_A$.

 % Since (\ref{22}) commutes with $\Lambda^n_A$, one immediately sees that the  with respect to either  $\{\textbf{s}, \textbf{s}'   \cdots, \textbf{s}^{(n)} \}$ or $\{\textbf{t}, \textbf{t}'   \cdots, \textbf{t}^{(n)} \}$,  (\ref{22}) represents the same linear map, which gives the following
%In particular,  the block matrix $\Lambda^n_A$  appearing in (\ref{trsf})   lies in $\Lambda^n$ whose first column is $A, A',\cdots A^{(n)}$.

\end{rem}
 \section{Pascal map on jet bundles}
 In Section 3.1 we give a brief introduction  on  jet bundles    and  in Section 3.2    we  introduce   a canonical  bundle map, called the \emph{Pascal map} on the jet bundles.
 \subsection{Jet bundle}

To   quickly  reveal  the   conceptual idea and make this note self-contained,    we   confine ourselves in  an  specific  setting  by focusing on      vector bundles     associated to   holomorphic   maps  from a domain in $\cc$   into a  complex Grassmannian, which is elementary and facilitates our presentations in  Section 4 as well. Readers familiar with abstract theory of jet bundles will find it routine to extend   discussions in this paper  to more  general settings(see Remark \ref{mafan}).

Let
  $Gr(l,\mathcal{H})$  be  the Grassmannian of $l$ dimensional subspaces in a complex vector space $\mathcal{H}$ and $f$  be   a   map  from      a domain $\Omega\subseteq\cc$   to $Gr(l,\mathcal{H})$. The map $f$  is     holomorphic in the sense  that
    for any point $ z_0$ in $\Omega$, there
 exists a neighborhood $\Delta$ of $ z_0$ and    holomorphic  $\mathcal {H}$-valued functions $s_1, \cdots,
 s_l$   on  $\Delta$ such that $$f( z)=\bigvee\{s_1( z), \cdots,
 s_l( z)\}$$ for every $ z\in\Delta$.
 The  vector bundle
 $$E:=\{(h,z)\in \mathcal {H}\times \Omega|h\in f(z)\}   $$ associated to $f$     is  then a holomorphic   vector bundle over $\Omega$ of rank $l$ and $\sss=\{s_1, \cdots,
 s_l\}$ implements a local holomorphic frame of $E$.

%\begin{rem} Note that  passing from   $f$  to $E$ is  essentially   an  ``extrinsic-intrinsic" transition, focusing on this model facilitates  coming discussions  in Section 3.\end{rem}
Fix   a positive integer $n$ and a point $z\in\Omega$, we set  \begin{equation}\label{jetf}E^n(z):=\bigvee_{1\leq i\leq l, 0\leq k\leq n} \{s_i^{(k)}(z)\},  \end{equation}    where $\textbf{s}=\{s_1, \cdots,
 s_l\}$ is any holomorphic frame of $E$ around $z$.

 \bigskip

\begin{Def}{The vector space  $E^n(z)$ does not depend on choice of $\sss$  hence $$E^n:=\{(h,z)\in H\times \Omega|h\in E^n(z)\}, $$ is a well-defined vector bundle, called the \textbf{$n$-jet bundle }of $E$. }\end{Def}
\bigskip

In fact, suppose  $\textbf{t}=\{t_1,\cdots t_l\}$ is another  holomorphic frame around $z$, then there is an invertible holomorphic  matrix function $A$ such that $\textbf{s}=A\textbf{t}$(here $\sss$ and $\textbf{t}$ are interpreted as column vectors so $A$ acts from the left). A   differentiation gives
\begin{equation}\label{trsf}\left(\begin{array}{c}{\textbf{s}}  \\ {\textbf{s}'} \\{\textbf{s}''}\\  {\vdots}  \\{\textbf{s}^{(n)}} \end{array}\right)=
\left(\begin{array}{cccccc}{A}  & { }  &{ } & { } & { }\\ {A'} & {A}  &  { } & {} & { }\\{A''} & {2A'} &  {A} & { } & { }\\  {\vdots} &{\vdots }  & { } & {}& {\ddots}& { } \\{A^{(n)}} & {\bin n 1 A^{(n-1)}}  &{\cdots} &{\cdots} & {\bin n {n-1}A'}& {A} \end{array}\right)\left(\begin{array}{c}{\textbf{t}}  \\ {\textbf{t}'} \\{\textbf{t}''}\\  {\vdots}  \\{\textbf{t}^{(n)}} \end{array}\right),\end{equation} hence  $\bigvee_{1\leq i\leq l, 0\leq k\leq n} \{s_i^{(k)}(z)\}$ and $\bigvee_{1\leq i\leq l, 0\leq k\leq n} \{t_i^{(k)}(z)\}$ is the same vector space as they just differs by an invertible block matrix.

The high order derivatives
 $\{s_i^{(k)}(z), 1\leq i\leq l, 0\leq k\leq n\}$(or $\{\sss,\sss'\cdots \sss^{(n)}\}$ for short) implements  a canonical local holomorphic frame for $E^n$.  An important fact as one immediately finds is that  the transition matrix (\ref{trsf}) between two such  frames   lies in the Pascal algebra $\Lambda^n$, which    is determined by its upper block   $A$   and  we   denote it  by $\Lambda_A^n$ in the sequel.
\subsection{Pascal map}
Now we are ready to introduce the  Pascal bundle  map  on $E^n$. Recall that a bundle map on a vector bundle maps each fiber linearly to itself,  and  given a collection of   frames $\{\sss_\alpha\}$ as local trivializations, the standard way to construct an unambiguously defined   bundle map $\Phi$  is to give a collection of matrix functions $\{\Phi(\sss_\alpha)\}$   such that the compatibility condition
\begin{equation}\label{mmmmcx}\Phi(\sss_\alpha)=A_{\alpha\beta}\Phi(\sss_\beta)A_{\alpha\beta}^{-1}\end{equation} holds, where $A_{\alpha\beta}$ is the transition function between $\sss_\alpha$ and $\sss_\beta$(so different matrices represents the same linear map). The following theorem asserts that on  the $n$-jet bundle of a holomorphic vector bundle,   a single constant matrix   will do, which gives the promised Pascal map on $E^n$.

\begin{thm}\label{pascc}Let $E$ be a holomorphic vector bundle over $\Omega\subseteq\cc$ and $n$ be a positive integer, the constant block  matrix (\ref{22})
%$$
%P=\left(\begin{array}{cccccc} {0}  & { }  &{ } & { } & { }\\ { I} & { 0}  &  { } & {} & { }\\
%{0} & {2I } &  { 0} & { } & { }\\{0} & {0} &  {3I } & {0} & { }\\  {\vdots} &{\vdots }  & { } & {}& {\ddots}& { } \\
%{0} & {  0 }  &{\cdots} &{\cdots} & {nI}& { 0 } \end{array}\right) $$
 represents  a well defined bundle map,  called the \textbf{Pascal map},    on $E^n$.
\end{thm}

\begin{proof}
From the construction of $E^n$  one sees that if a collection of  frames  $\{\sss_\alpha\}$ gives   local trivializations for $E$, then $\{\textbf{s}_\alpha, \textbf{s}'_\alpha   \cdots, \textbf{s}^{(n)}_\alpha \}$ gives local trivializations for $E^n$. If $\textbf{s}$ and $\textbf{t}$  are any two overlapping holomorphic  frames of $E$ with transition matrix $A$, then the transition matrix for $\{\sss,\sss'\cdots \sss^{(n)}\}$ and $\{\textbf{t},\textbf{t}'\cdots \textbf{t}^{(n)}\}$ is $\Lambda_A^n$(see (\ref{trsf}) above). Now it suffices to verify the compatibility condition
  $P=\Lambda_A^nP(\Lambda_A^n)^{-1}$, or equivalently,  $P\Lambda_A^n=\Lambda_A^nP$.  As $\Lambda_A^n$ lies in the Pascal algebra,  this follows from  the block matrix version of Theorem \ref{pascal}.

\end{proof}

%
%\begin{Def}\label{depascal}
%Let $E$ be a holomorphic vector bundle over $\Omega\subseteq\cc$ and $n$ be a positive integer, the bundle map  defined by the constant block matrix  (\ref{22})  is on $E^n$.
%\end{Def}
Explicitly, for  \emph{any } holomorphic frame $\textbf{s}=\{s_1,\cdots s_l\}$ of $E$, the Pascal map(still denoted by $P$)  acts on  the local frame  $\{\textbf{s}(z), \textbf{s}'(z),\cdots \textbf{s}^{(n)}(z)\}$ of $E^n$ by
\begin{equation}\label{pori} Ps_i^{(k)}=ks_i^{(k-1)}, 1\leq k\leq n, \quad \mathrm{and} \quad Ps_i=0\end{equation} for all $1\leq i\leq l$.
 %Throughout this paper, ``holomorphic   vector bundle" or ``holomorphic Hermitian vector bundle"  will always mean vector bundles associated to holomorphic maps into $Gr(l,\mathcal{H})$,  and it will be seen that for general holomorphic vector bundles our discussions  still  work(see Remark \ref{mafan} below).
\begin{rem}\label{mafan}On a general holomorphic vector bundle
%An obvious advantage of working with    bundles associated to holomorphic maps into Grassmannian is that the sections of $E$   are vector-valued holomorphic functions hence one can   taking derivatives of them to construct   $E^n$ is a quite straightforward way. Moreover,  this model exists naturally in operator theory on Hilbert spaces to which our study has a close relation  as will be seen in Section 4.
  where it is not so straightforward  to make sense of  high order derivative  of its sections,     $E^n$  is standardly  defined  via transition functions with compatibility conditions. Precisely, let $\{U_\alpha\}$ be  local trivializations  of $E$ and  $\{A_{\alpha\beta}\}_{U_\alpha\cap U_\beta\neq\emptyset}$ be the   corresponding set of  transition matrix functions,  then it is easy to  check that
  the matrix functions   $\{\Lambda^n_{A_{\alpha\beta}}\}_{U_\alpha\cap U_\beta\neq\emptyset}$   satisfy the compatibility condition $\Lambda^n_{A_{\alpha\beta}}\Lambda^n_{A_{\beta\gamma}}\Lambda^n_{A_{\gamma\alpha }}=I$ when $U_\alpha\cap U_\beta\cap U_\gamma\neq\emptyset$ hence well defines a vector bundle which is the     $n$-jet bundle $E^n$. In particular, our construction of the Pascal map above remains valid since it only depends on the fact that $\Lambda^n_{A_{\alpha\beta}}$ commutes with (\ref{22}).  \end{rem}
\bigskip

Like many familiar    geometric notions(such as   curvature or torsion),  the intrinsically defined  Pascal map   also has  an ``extrinsic predecessor".     Precisely, if      there exists   a bounded linear operator $T$ on the   ambient space $\mathcal {H}$  extending  ``coordinate   multiplication" on $E$, that is,
%Now suppose $E$ is a vector bundle associated to a holomorphic map  into $Gr(l, \mathcal{H})$  and there is a bounded linear operator $T$ on $\mathcal{H}$ such that for every $z\in\Omega$, it holds that
$Ts_i(z)=zs_i(z), 1\leq i\leq l$, then differentiating  it  $k$ times yields
$(T-z)s_i^{(k)}(z)= ks_i^{(k-1)}(z), $ which is  exactly the action     (\ref{pori}) of the Pascal map. Boundedness of $T$   validating  this differentiation argument  is a nontrivial extrinsic  condition and this operator is historically  called the Cowen-Douglas operator\cite{CD}. Our study in the next section will involve  unitary equivalence of Pascal maps,  which also has an  extrinsic counterpart as  unitary equivalence of  operators restricted to generalized eigen-spaces(see Sec 2.\cite{CD}),   hence  our study   can be somehow regarded as an ``intrinsic Cowen-Douglas theory".

\section{Contact   between vector bundles}
Contact  order  between smooths maps    is a useful extrinsic invariant. Historically it   implements a universal criteria  for the   congruence   problem of determining if two maps can be identified up to a rigid motion in the ambient spaces(Sec.5 \cite{Gri}). In particular,  the congruence problem for holomorphic maps into complex  Grassmannians had been extensively studied  over years(\cite{CZ,Gri,JP1,JP2}).
 In this section we   refine this classical notion  into an frame-independent  version  so that it makes sense on the associated vector bundles.
 To make sense of ``rigid motion",  we   assume in this section  that the ambient space $\mathcal{H}$ admits an inner product(Hilbert space) and  correspondingly $E$ and $E^n$  are Hermitian vector bundles.

Given two holomorphic maps $f$ and $\w f$ from a domain $\Omega\subseteq\cc$ into $Gr(l, \mathcal {H})$,   point-wise contact between them are defined to be    order $n$ agreement    up to a  rigid motion(isometry) of the ambient space(see Sec.5,\cite{Gri} or Sec.2,\cite{CD}):
 \begin{Def}\label{cont}Let $\mathcal {H}$ be a Hilbert space.
 Two holomorphic maps $f$ and $\w f$ from a domain $\Omega\subseteq\cc$ into $Gr(l, \mathcal {H})$  are said to have {contact of order $n$} at a point  $z_0$ if there exist holomorphic  frames $\textbf{s}=\{s_1, \cdots,
 s_l\}$ and $\widetilde{\textbf{s}}=\{\w s_1, \cdots,
 \w s_l\}$ for $f$ and $\w f$ around $z_0$   such that the linear map  defined by
$$s_i^{(k)}(z_0)\mapsto \w s_i^{(k)}(z_0), 1\leq i\leq l, 0\leq k\leq n$$ is isometric.\end{Def}

In above definition one have  to assume that the holomorphic maps are \textbf{$n$-nondegenerate}, that is,  $  \{s_i^{(k)}(z),1\leq i\leq l, 0\leq k\leq n\}$ and $\{\w s_i^{(k)}(z),1\leq i\leq l, 0\leq k\leq n\}$  are    linearly independent sets. We keep this assumption in the sequel.

\begin{prop}\label{contin}Let $\mathcal {H}$ be a Hilbert space. Let  $f$ and $\w f$  be holomorphic maps   from $\Omega\subseteq\cc$ to $Gr(l, \mathcal{H})$ with associated  holomorphic Hermitian  vector bundles  $E$ and $\w E$. Fix  a point   $z_0$  in $\Omega$, the followings are equivalent:

(i) $f$ and $\w f$ have contact of order $n$ at $z_0.$

%(ii) For certain(any) holomorphic frames $\textbf{s}=\{s_1,\cdots s_l\}$ and $\widetilde{\textbf{s}}=\{\w s_1,\cdots \w s_l\}$ of $E$ and $\w E$ around $z_0$,  there is  a block matrix in $\Lambda^n$  which represents a linear  isometric map  from $E^n(z_0)$ to $\w E^n(z_0)$ with respect to  $\{\textbf{s}(z_0), \textbf{s}'(z_0)   \cdots, \textbf{s}^{(n)} (z_0)\}$ and $\{ \widetilde{\textbf{s}}(z_0), \widetilde{\textbf{s}}'(z_0)   \cdots,  \widetilde{\textbf{s}}^{(n)}(z_0) \}$.
% (ii)For  certain holomorphic frames  $\textbf{s} $ and $\widetilde{\textbf{s}}$ of $E$ and $\w E$, there exists an linear isometry  $\Phi$ from $E^n(z_0)$ to $\w E^n(z_0)$ which   can be represented by a  block matrix in  $\Lambda^n$ with respect to $\{\textbf{s}(z_0), \textbf{s}'(z_0)   \cdots, \textbf{s}^{(n)} (z_0)\}$ and $\{ \widetilde{\textbf{s}}(z_0), \widetilde{\textbf{s}}'(z_0)   \cdots,  \widetilde{\textbf{s}}^{(n)}(z_0) \}$.

 (ii) There is a linear isometric map   $\Phi$ from $E^n(z_0)$ to $\w E^n(z_0)$ such that $\Phi P=\w P\Phi$, where $P$ and $\w P$ are Pascal maps on  $E^n$ and $\w E^n$.

%$$
%\left(\begin{array}{cccccc}{A_0}  & { }  &{ } & { } & { }\\ {A_1} & {A_0}  &  { } & {} & { }\\{A_2} & {2A_1} &  {A_0} & { } & { }\\{A_3} & {3A_2} &  {3A_1} & {A_0 } & { }\\  {\vdots} &{\vdots }  & { } & {}& {\ddots}& { } \\{A_n} & {\bin n 1 A_{n-1}}  &{\cdots} &{\cdots} & {\bin n {n-1}A_1}& {A_0 } \end{array}\right)$$
 %with respect to   $\{\textbf{s}(z_0), \textbf{s}'(z_0)  \cdots, \textbf{s}^{(n)}(z_0)\}$ and $\{\widetilde{\textbf{s}}(z_0), \widetilde{\textbf{s}}'(z_0),\cdots, \widetilde{\textbf{s}}^{(n)}(z_0)\}$.
%
%(iii) Statement   (ii) above  holds for any holomorphic frames of $E$ and $\w E$.
%

\end{prop}

\begin{proof}

  (i)$\Rightarrow$(ii)Let  $\Phi$ be the  isometric linear map  as in Definition \ref{cont}, then   $\Phi P=\w P\Phi$ holds trivially since $P$ and $\w P$ are both represented by (\ref{22}) while $\Phi$ is represented by   the identity matrix.

(ii)$\Rightarrow$(i)  Fix holomorphic frames $\sss$ and $\w\sss$ for $E$ and $\w E$,     $\Phi P=\w P\Phi$  implies that  the  representing matrix of $\Phi$ commutes with (\ref{22}) with respect to  $\{\textbf{s}(z_0), \textbf{s}'(z_0)   \cdots, \textbf{s}^{(n)} (z_0)\}$ and $\{ \widetilde{\textbf{s}}(z_0), \widetilde{\textbf{s}}'(z_0)   \cdots,  \widetilde{\textbf{s}}^{(n)}(z_0) \}$.   By Theorem \ref{pascal},  the matrix has to be in $\Lambda^n$ hence
 there exists matrices $A_0, A_1,\cdots A_n$ such that   the linear isometric   map  $\Phi$ is represented  by

 $$ \left(\begin{array}{c}{\textbf{s}(z_0)}  \\ {\textbf{s}'(z_0)} \\{\textbf{s}''(z_0)}\\  {\vdots}  \\{\textbf{s}^{(n)}(z_0)} \end{array}\right)\mapsto
\left(\begin{array}{cccccc}{A_0}  & { }  &{ } & { } & { }\\ {A_1} & {A_0}  &  { } & {} & { }\\{A_2} & {2A_1} &  {A_0} & { } & { }\\  {\vdots} &{\vdots }  & { } & {}& {\ddots}& { } \\{A_{n}} & {\bin n 1 A_{n-1}}  &{\cdots} &{\cdots} & {\bin n {n-1}A_1}& {A_0} \end{array}\right)\left(\begin{array}{c}{\widetilde{\textbf{s}}(z_0)}  \\ {\widetilde{\textbf{s}}'(z_0)} \\{\widetilde{\textbf{s}}''(z_0)}\\  {\vdots}  \\{\widetilde{\textbf{s}}^{(n)}(z_0)} \end{array}\right).$$
     Set $A(z)=\sum_{k=0}^n \frac{1}{k!}(z-z_0)^kA_k$, then  $A(z)$ is a holomorphic matrix function around $z_0$ with $A^{(k)}(z_0)=A_k$,  $k=0,1\cdots n$. Let $\widetilde{\textbf{t}}=A(z)\widetilde{\textbf{s}}$, then   $\widetilde{\textbf{t}}$ is also a holomorphic frame for $\w E$(over a sufficiently small neighborhood of $z_0$) and the frames $ {\textbf{s}} $  and $\widetilde{\textbf{t}}$  meets Definition \ref{cont}.

\end{proof}

Now we see that condition $(ii)$ of  Proposition \ref{contin} as a frame-independent criteria        is compatible with the original extrinsic Definition \ref{cont}, hence   is eligible to be the intrinsic condition that  defines point-wise contact between holomorphic Hermitian vector bundles over  a one dimensional domain:

\begin{Def}\label{defcon}Two   holomorphic Hermitian vector bundles $E$ and $\w E$ over $\Omega\subseteq\cc$  are said to have contact of order $n$ at a point $z_0$ if  there is a linear isometric map   $\Phi$ from $E^n(z_0)$ to $\w E^n(z_0)$ such that $\Phi P=\w P\Phi$, where $P$ and $\w P$ are Pascal maps on  $E^n$ and $\w E^n$.\end{Def}
 Similar to Remark \ref{mafan}, in Definition \ref{defcon} one do not need to assume the vector bundles are associated to holomorphic maps. In fact,  if $E$   a general Hermitian vector bundle where   $\{H_\alpha\}$ is the Gram matrix for the local frame on $U_\alpha$, then one can check that $\{[\p^p\overline{\p^q} H_\alpha]_{0\leq p,q\leq n}\}$ also glue to a well defined Hermitian form on $E^n$ hence it makes sense to talk about isometric bundle maps on the jet bundles.
%\begin{rem}Extrinsically,    $\Phi P=\w P\Phi$ corresponds to unitary equivalence restriction of Cowen-Douglas operators on generalized eigen-spaces(see Sec.2 \cite{CD})\end{rem}

\section{Several variable case}
In  this section we outline   how to extend discussions in Section 3 and  Section 4  with  $\Omega\subseteq\cc^m$, $m>1$. Such a several variable extension   is not trivial but  given the idea in previous sections on  $m=1$, this is essentially  a  technical work and we omit the details.

We begin with jet bundles and the Pascal map. Fix a   point $\zz=(z_1,\cdots z_m)\in\Omega$ and a  holomorphic frame $ {\textbf{s}}=\{s_1\cdots s_l\} $ of $E$ around $\zz$, set

\begin{equation}\label{jetm}E^n(\zz):=\bigvee_{1\leq i\leq l, 0\leq |I|\leq n}\partial^Is_i(\zz), \end{equation}
where $I=(i_1,\cdots i_m)$ is an multi-index and $|I|=i_1+\cdots+i_m$.

 For any $1\leq k\leq m$, we  define  a linear map on $E^n(\zz)$  by

 \begin{equation}\label{diffe} \partial^Is_i(\zz)\mapsto \left\{\begin{array}{l}{i_k\partial_{z_1}^{i_1}\cdots\partial_{z_k}^{i_k-1}\cdots\partial_{z_m}^{i_m}s_i(\zz), i_k\geq 1} \\ {0, \quad\quad  {i_k=0}}\end{array}\right.\end{equation}
where $\partial^I=\partial_{z_1}^{i_1}\cdots\partial_{z_m}^{i_m}, 1\leq i\leq l$.

One need  to check two issues:

 $(i)$ if $ {\textbf{t}}=\{t_1,\cdots t_l\}$ is another holomorphic frame of $E$ around $\zz$, then
 $\bigvee_{1\leq i\leq l, 0\leq |I|\leq n}\partial^Is_i(\zz)$ and $\bigvee_{1\leq i\leq l, 0\leq |I|\leq n}\partial^It_i(\zz)$  is the same vector space so     the $n$-jet bundle $E^n$ of $E$ is well-defined;

 $(ii)$ with respect to either
 $\{\partial^Is_i(\zz), 1\leq i\leq l, 0\leq |I|\leq n\}$  or $\{\partial^It_i(\zz), 1\leq i\leq l, 0\leq |I|\leq n\}$, the rule (\ref{diffe}) represents the same linear map on $E^n(\zz)$, which is the several varaible analogue of Theorem \ref{pascc}.

 With these issues checked, one sees that for every $1\leq k\leq m$, (\ref{diffe}) gives a well defined bundle map, called the \textbf{$\textbf{k}$-th Pascal map} on $E^n$(denoted by $P_k$).
 \bigskip

Finally, one can prove the following analogue  of        Proposition \ref{contin} where condition $(ii)$ implements the intrinsic definition of  point-wise contact between Hermitian holomorphic vector bundles in the several variable case.
 \begin{prop}\label{continss}Let $\mathcal {H}$ be a Hilbert space,   $f$ and $\w f$  be holomorphic maps   from $\Omega\subseteq\cc^m$ to $Gr(l, \mathcal{H})$  with associated    holomorphic Hermitian  vector bundles   $E$ and $\w E$. The followings are equivalent

(i)  $f$ and $\w f$ have contact of order $n$ at $z_0,$ that is, there exists holomorphic frames $\textbf{s}=\{s_1,\cdots s_l\}$ and  $\widetilde{\textbf{s}}=\{\w s_1,\cdots \w s_l\}$   around $\zz_0$ such that the linear map defined by $$\partial^Is_i(\zz_0)\mapsto \partial^I\w s_i (\zz_0), 1\leq i\leq l, 0\leq |I|\leq n$$ is isometric.

(ii)  There is a linear isometric map   $\Phi$ from $E^n(\zz_0)$ to $\w E^n(\zz_0)$ such that $\Phi P_k=\w P_k\Phi$ for all  $1\leq k\leq m$, where $P_k$ and $\w P_k$ are $k$-th Pascal maps on  $E^n$ and $\w E^n$.

  \end{prop}

\quad

SCHOOL OF MATHEMATICS, SHANDONG UNIVERSITY, JINAN 250100, CHINA.

\footnotesize \emph{Email:} \textbf{ lchencz@sdu.edu.cn}
\end{document}